\documentclass[11pt,a4paper,reqno]{amsart}
\usepackage{amsfonts}
\usepackage{amsmath}
\usepackage{amssymb}
\usepackage{amsthm}
\usepackage{amsxtra}
\usepackage{epic, eepic}
\usepackage{epsfig,color}
\usepackage{fullpage}
\usepackage{graphicx}
\usepackage{subfigure}
\DeclareFontFamily{U}{shuffle}{}
\DeclareFontShape{U}{shuffle}{m}{n}{%
<5-8>shuffle7%
<8->shuffle10%
}{}

\DeclareSymbolFont{Shuffle}{U}{shuffle}{m}{n}
\DeclareMathSymbol\shuffle{\mathbin}{Shuffle}{"001}

\theoremstyle{plain}
\newtheorem*{thm:ADinBA}{Theorem \ref{thm:ADinBA}}
\newtheorem*{thm:main}{Theorem \ref{thm:main}}
\newtheorem*{thm:symmetries}{Theorem \ref{thm:symmetries}}

\newtheorem{Thm}{Theorem}[section]
\newtheorem{Cor}[Thm]{Corollary}

\newtheorem{Prop}[Thm]{Proposition}
\newtheorem*{Claim}{Claim}

\theoremstyle{definition}

\theoremstyle{remark}
\newtheorem{Rem}{Remark}[section]

\newtheorem{Ex}{Example}[section]

\definecolor{violet}{RGB}{204,52,255}

\numberwithin{equation}{section}

\newcommand{\ptd}{\mathsf{ptd}}
\newcommand{\murder}{\digamma}

\begin{document}
\title{Statistics on parallelogram polyominoes and a $q,t$-analogue of the Narayana numbers}
\author[J-C Aval]{Jean-Christophe Aval$^{*}$} \thanks{$^{*}$ Supported by ANR -- PSYCO project
(ANR-11-JS02-001).}
\address{LaBRI, Universit\'e de Bordeaux, CNRS, 351 cours de la Lib\'eration, 33405 Talence, France}\email{aval@labri.fr}\email{borgne@labri.fr}
\author[M. D'Adderio]{Michele D'Adderio}
\address{Universit\'e Libre de Bruxelles (ULB)\\D\'epartement de Math\'ematique\\ Boulevard du Triomphe, B-1050 Bruxelles\\ Belgium}\email{mdadderi@ulb.ac.be}
\author[M. Dukes]{\\Mark Dukes}
\address{University of Strathclyde\\Department of Computer and Information Sciences\\ 16 Richmond Street, Glasgow G1 1XQ\\ Scotland, United Kingdom}\email{mark.dukes@strath.ac.uk}
\author[A. Hicks]{Angela Hicks$^{\dag}$} \thanks{$^{\dag}$ Supported by NSF grant DMS=0800273.}
\address{UCSD\\ Department of Mathematics\\ 9500 Gilman Drive\\ 92093-0112 La Jolla, USA}\email{ashicks@math.ucsd.edu}
\author[Y. Le Borgne]{Yvan Le Borgne}


\begin{abstract}
We study the statistics $\mathsf{area}$, $\mathsf{bounce}$ and
$\mathsf{dinv}$ on the set of parallelogram polyominoes having a
rectangular $m$ times $n$ bounding box.
We show that the bi-statistics $(\mathsf{area},\mathsf{bounce})$ and
$(\mathsf{area},\mathsf{dinv})$ give rise to the same
$q,t$-analogue of Narayana numbers which was introduced by two of
the authors in \cite{dukesleborgne}.
We prove the main conjectures of that paper:
the $q,t$-Narayana polynomials are symmetric in both $q$ and $t$, and $m$ and $n$.
This is accomplished by providing a symmetric functions interpretation of the $q,t$-Narayana polynomials
which relates them to the famous diagonal harmonics.
\end{abstract}

\maketitle

\section{Introduction}
\label{sec:one}
A {\it{parallelogram polyomino}} having an $m\times n$ {\it{bounding box}} is a polyomino in a rectangle consisting of $m\times n$ cells
that is formed by cutting out two (possibly empty) non-touching standard Young tableaux which have corners at $(0,n)$ and $(m,0)$.
An example of a parallelogram polyomino having a $12\times 7$ bounding box is illustrated in Figure \ref{fig1}.
Let $\mathrm{Polyo}_{m,n}$ be the set of all {parallelogram polyominoes} having a rectangular $m\times n$ bounding box.
The cardinality of $\mathrm{Polyo}_{m,n}$ is known to be
$N(m+n-1,m)$ where for positive integers $a$ and $b$,
$$N(a,b):=\frac{1}{a}{a \choose b} {a \choose b-1}$$
are the famous \textit{Narayana numbers}.
Two authors of this work introduced \cite{dukesleborgne}
two statistics on these combinatorial objects, $\mathsf{area}$ and
$\mathsf{bounce}$, which led to a $q,t$-analogue of the Narayana
numbers $N(m+n-1,m)$, namely
$$
\mathsf{Nara}_{m,n}(q,t):=\sum_{P
\in \mathrm{Polyo}_{m,n}
}q^{\mathsf{area}(P)}t^{\mathsf{bounce}(P)},
$$
that they called the $q,t$-Narayana polynomial.
In that same work it was conjectured that the $q,t$-Narayana polynomials are
symmetric in $q$ and $t$, and as expressions were also symmetric
in $m$ and $n$.
We introduce a new statistic $\mathsf{dinv}$ which
gives a new $q,t$-analogue of the same numbers
$$
\widetilde{\mathsf{Nara}}_{m,n}(q,t):=\sum_{P
}q^{\mathsf{dinv}(P)}t^{\mathsf{area}(P)}.
$$

The following theorem establishes a relation between these two
polynomials.

\begin{thm:ADinBA}
For all $m\geq 1$ and $n\geq 1$, we have
$$
\mathsf{Nara}_{m,n}(q,t)=\widetilde{\mathsf{Nara}}_{n,m}(q,t).
$$
\end{thm:ADinBA}

We give two proofs of this result, one using an explicit
bijection, and another one using a recursion.
The main result of this paper is the proof of the symmetries
conjectured in \cite{dukesleborgne}.

\begin{thm:symmetries}
For all $m\geq 1$ and $n\geq 1$, we have
$$ \mathsf{Nara}_{m,n}(q,t)=\mathsf{Nara}_{m,n}(t,q) $$
and
$$ \mathsf{Nara}_{m,n}(q,t)=\mathsf{Nara}_{n,m}(q,t).  $$
In particular
$$ \mathsf{Nara}_{m,n}(q,t)=\widetilde{\mathsf{Nara}}_{m,n}(q,t).  $$
\end{thm:symmetries}
In order to prove this result, we will give a symmetric functions
interpretation of our $q,t$-Narayana numbers:
\begin{thm:main}
For all $m\geq 1$ and $n\geq 1$ we have
$$
\mathsf{Nara}_{m,n}(q,t)=(qt)^{m+n-1}\cdot \langle \nabla
e_{m+n-2},h_{m-1} h_{n-1} \rangle ,
$$
where $e_k$ and $h_k$ are the elementary and the homogeneous
symmetric functions of degree $k$ respectively, $\nabla$ is the
well known nabla operator introduced by Bergeron and Garsia (see
\cite[Section 9.6]{bergeron}), and the scalar product is the usual
Hall inner product on symmetric functions.
\end{thm:main}
This result establishes a remarkable link between the
$q,t$-Narayana polynomials and the well-known \textit{diagonal
harmonics} $DH_n$, since $\nabla e_n$ is the Frobenius
characteristic of this important module of the symmetric group
$\mathfrak{S}_n$, as shown by Haiman in \cite{haiman}.

Haglund \cite{haglund} gave a combinatorial interpretation of the
polynomial $\langle \nabla e_{m+n-2},h_{m-1} h_{n-1} \rangle$ in
terms of parking functions.
In fact Haglund's result would be an
easy consequence of the famous \textit{shuffle conjecture}, which
predicts a combinatorial interpretation of $\nabla e_n$ in terms
of parking functions (see \cite[Chapter 6]{haglundbook}), if a proof of
it could be found.

In order to prove Theorem \ref{thm:main}, we use the results of
Section \ref{sec:five}, proving that the combinatorial polynomials in 
Haglund's result and our $q,t$-Narayana polynomials
both satisfy the same recursion.
This paper is organized in the following way:
\begin{itemize}
\item In Section \ref{sec:two} we define three statistics on parallelogram polyominoes
    and two $q,t$-analogues of Narayana numbers.
\item In Section \ref{sec:three} we establish a bijection between our
parallelogram polyominoes and a set of Dyck paths. We classify those words that are area words of members of $\mathrm{Polyo}_{m,n}$. Area words are important in the definition of several statistics mentioned in Section \ref{sec:two}.
    \item In Section \ref{sec:four} we present a bijection from
    $\mathrm{Polyo}_{m,n}$ to $\mathrm{Polyo}_{n,m}$ which sends the
    bi-statistic $(\mathsf{area},\mathsf{bounce})$ to the
    bi-statistic $(\mathsf{dinv},\mathsf{area})$, thereby establishing
    Theorem \ref{thm:ADinBA}.
    \item In Section \ref{sec:five} we prove a recursion satisfied by
    both of our $q,t$-Narayana polynomials, which gives another proof of Theorem
    \ref{thm:ADinBA}.
    \item In Section \ref{sec:six} we provide the necessary background to
    state Theorem \ref{thm:main}, and we show how Theorem \ref{thm:symmetries} follows from it.
    Theorem \ref{thm:main} is then proven.
\end{itemize}

\section{Three statistics on parallelogram polyominoes}
\label{sec:two}
We may give an alternative characterisation of parallelogram polyominoes in terms of non-intersecting paths in the plane.
This alternative charaterisation will prove useful in defining the statistics and mappings used in the paper.

Consider a rectangular grid in $\mathbb{Z}^2$ of width $m$ and
height $n$. On this grid consider two paths, both starting from
the Southwest corner and arriving at the Northeast corner,
travelling on the grid, performing only North or East steps, with
the further restriction that they touch each other only at the
starting point and at the ending point. Such a pair of paths
uniquely defines a parallelogram polyomino. The region between the
two paths is called the \emph{interior} of the (parallelogram)
polyomino. The two paths defining the parallelogram polyomino of
Figure \ref{fig1} are coloured in red and green, and the interior
has been shadowed.

\begin{figure}[h]
\includegraphics[scale=0.5]{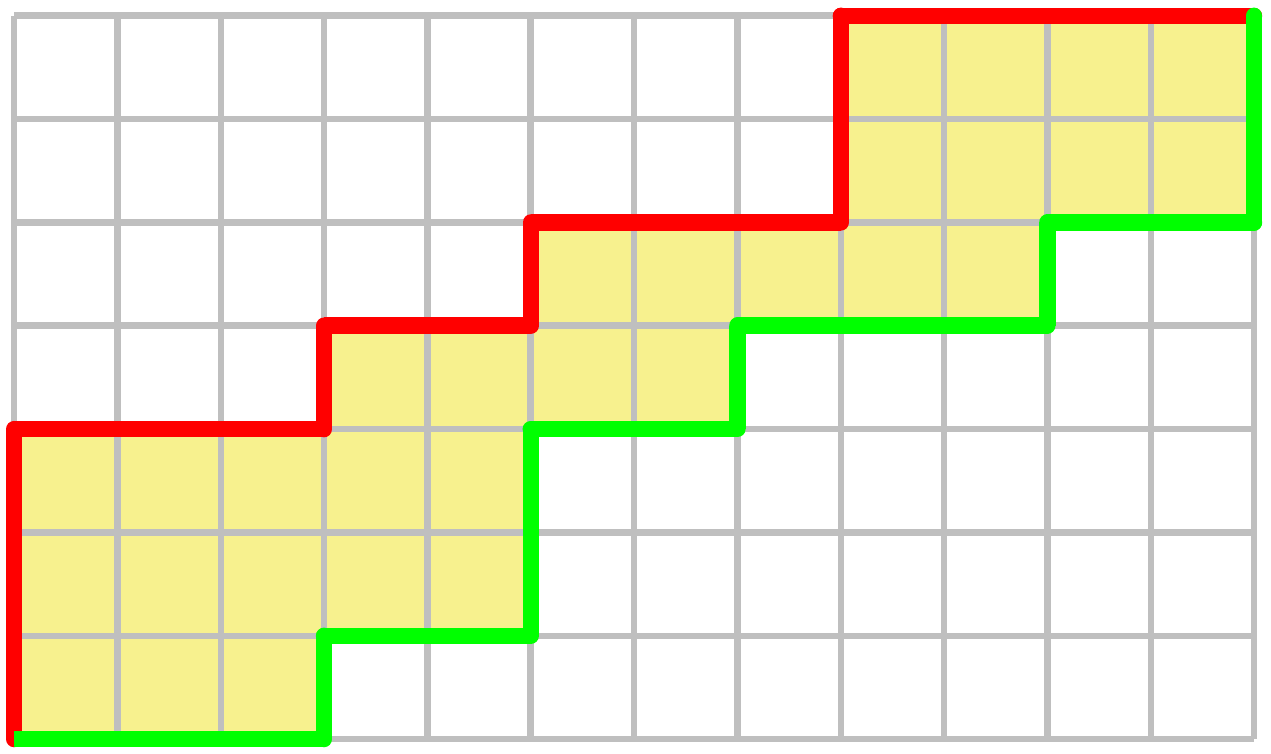}
\caption{\label{fig1}
A parallelogram polyomino having a $12$ times $7$ bounding box.}
\end{figure}

In what follows we will encode a parallelogram polyomino as an \emph{area word} consisting of
natural numbers (\emph{unbarred numbers}) and natural numbers with a bar on top (\emph{barred numbers}), in the
following way.
We will label every North step of the upper (red) path with a barred
number, and every East step of the lower (green) path
with an unbarred number. This is done in two stages.

First, for each East step of the lower path we draw a line
starting with the East endpoint and going Northwest until reaching
the upper path: we label this step with the number of squares
crossed by this line.
Second, we label each North step of the upper path with the number
of squares in the interior of the polyomino to the East of it
which were not crossed by any of the lines that we drew during the
previous stage.
An example of this labelling is shown in Figure \ref{fig2}, where we put a
black dot in the non-crossed squares.

\begin{figure}[h]
\includegraphics[scale=0.5]{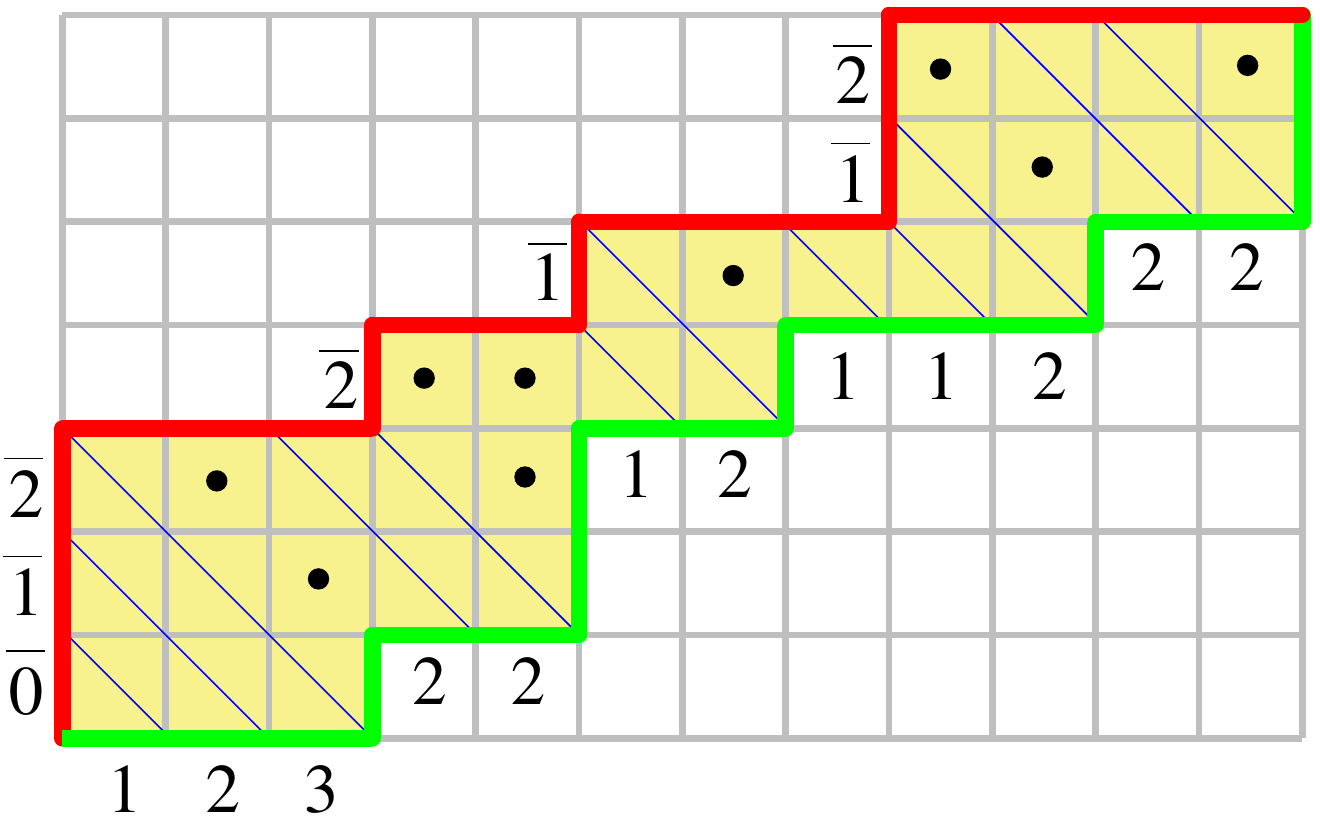}
\caption{\label{fig2}
The parallelogram polyomino of Figure~\ref{fig1} with its perimeter labelled.}
\end{figure}

Once we have done this labelling, we read the labels in the
following order: starting from Southwest and going to Northeast
imagine moving a straight line of slope $-1$ over the polyomino.
When we encounter vertical steps of the upper path or horizontal
steps of the lower path we write the corresponding labels. If we
encounter both types of steps at the same time then we write the
label of the upper path first. The area word of the example in
Figure \ref{fig2} is
$\overline{0}1\overline{1}2\overline{2}322\overline{2}1\overline{1}211\overline{1}2\overline{2}22$.

Notice that the sum of these numbers (disregarding the bars) gives
the $\mathsf{area}$ of the polyomino, which is the number of
squares between the two paths. This is the first of the statistics
that are relevant to us. In the example the $\mathsf{area}$ is 30.

Next we will define the $\mathsf{dinv}$ statistic.
Consider the total order on the labels
$$\overline{0}<1<\overline{1}<2<\overline{2}<3<\overline{3}<4<\overline{4}<\cdots.$$
Given a polyomino with area word $a_1a_2\dots a_k$, we define its
$\mathsf{dinv}$ as the number of pairs $a_i,a_j$ with $i<j$ and
$a_j$ is the immediate successor of $a_i$ in the fixed order. In
the example of Figure \ref{fig2}, the number of such pairs
containing $\overline{0}$ is $4$ and the $\mathsf{dinv}$ of the
polyomino is 35.

The last statistic that we introduce is the $\mathsf{bounce}$.
Consider the following path in a given polyomino: begin with a
single East step from the Southwest corner, and then move North
until reaching the East endpoint of a horizontal step of the upper
path; at this point we ``bounce'', i.e. we start moving East,
until we reach the North endpoint of a vertical step of the lower
path; at this point we ``bounce'' again, start moving North, and
we repeat this procedure until we reach the Northeast corner. This
path is called the {\it{bounce path}}.

Once we have the bounce path, starting from Southwest corner, we
label each step of the first sequence of vertical steps with $1$,
then each step of the second of such sequences with $2$, and so
on; we label each step of the first sequence of horizontal steps
with $\overline{0}$, then each step of the second of such
sequences with $\overline{1}$, and so on. See Figure \ref{fig3}
for an example of this labelling.

\begin{figure}[h]
\includegraphics[scale=0.5]{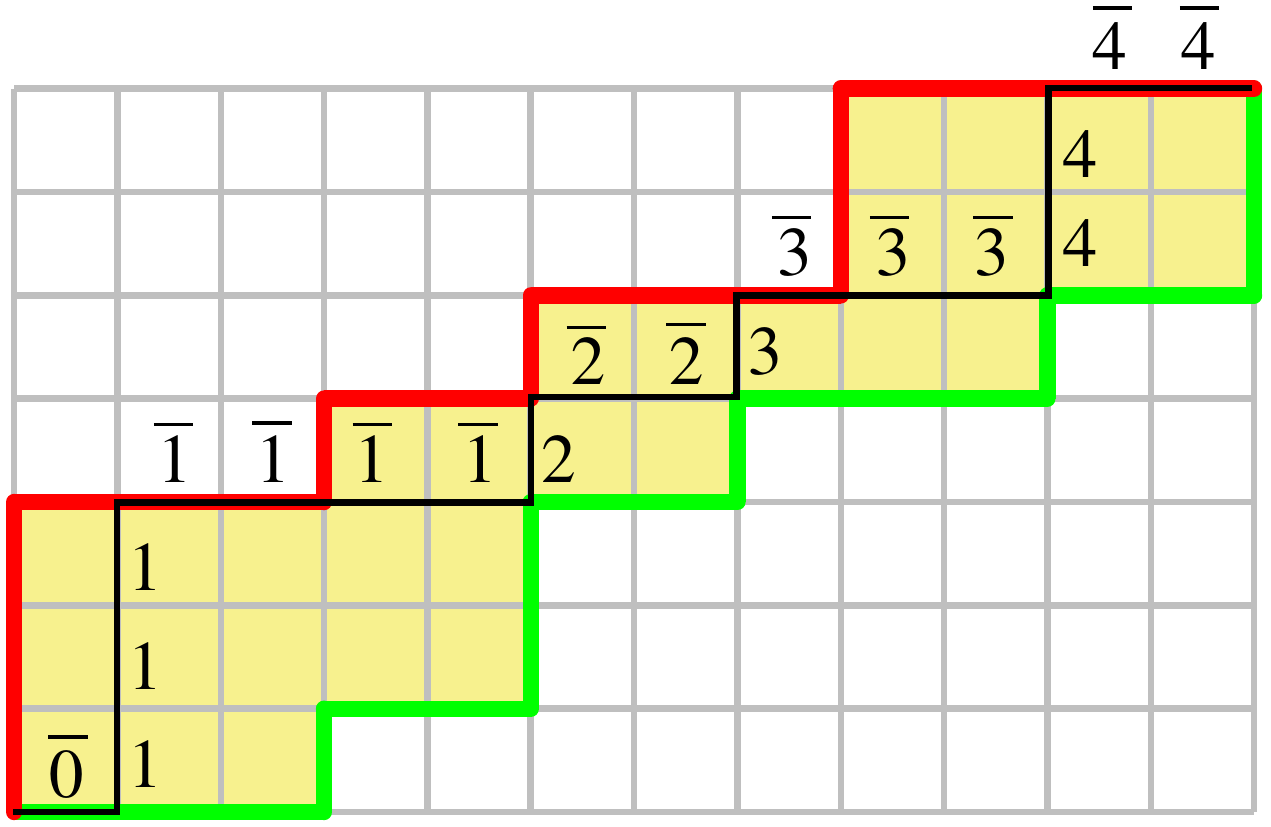}
\caption{\label{fig3}
The labelled bounce path.}
\end{figure}

The $\mathsf{bounce}$ of a polyomino is the sum of
the labels of its bounce path, disregarding the bars.
The $\mathsf{bounce}$ of the parallelogram polyomino in Figure \ref{fig3} is 41.

These three statistics give rise to a pair of bi-statistics on $\mathrm{Polyo}_{m,n}$ whose generating functions
$$ \mathsf{Nara}_{m,n}(q,t):=\sum_{P\in
\mathrm{Polyo}_{m,n}}t^{\mathsf{area}(P)}q^{\mathsf{dinv}(P)} $$
and
$$ \widetilde{\mathsf{Nara}}_{m,n}(q,t):=\sum_{P\in
\mathrm{Polyo}_{m,n}}t^{\mathsf{bounce}(P)}q^{\mathsf{area}(P)}.  $$
are studied in this paper.
The polynomials $\mathsf{Nara}_{m,n}(q,t)$ where first introduced
in \cite{dukesleborgne} by two of the authors of the present work.
In the same paper, it was conjectured that these were polynomials
symmetric in $q$ and $t$, and as expressions symmetric in $m$ and
$n$.

\section{A bijection with Dyck paths}
\label{sec:three} In this section we present a bijection $\ptd$
between the set $\mathrm{Polyo}_{m,n}$ and a set of Dyck paths
having length $2(m+n)$. We then prove a result which shows how to
read the area word of a parallelogram polyomino from its
corresponding Dyck path under $\ptd$. From this we will get a
characterization of the area words of polyominoes from
$\mathrm{Polyo}_{m,n}$ which will be used in the proof of Theorem
\ref{thm:ADinBA}.  We finally observe that this description
provides a way to computationally work with the set of area words
of $\mathrm{Polyo}_{m,n}$ by working with the easier to construct
set of Dyck paths.

Recall that a \textit{Dyck path} can be thought of as a path
consisting of Northeast or Southeast steps lying between parallel
horizontal lines, such that the path starts with a Northeast step,
it never crosses the starting horizontal line, and returns to it
at the end. Its \textit{length} is simply the number of its steps
it contains. Figure \ref{fig4} shows an example of a Dyck path
having length 38.

Notice that a Dyck path is uniquely determined by the sequence of rises and
falls we encounter as we move along the path from left to right.

We will next describe a bijection between the polyominos in $\mathrm{Polyo}_{m,n}$
and the set of Dyck path of length $2(m+n)$ with
$m$ rises in even positions and $n$ rises in odd positions, which
do not return to the starting horizontal line until the end.
This bijection appears in \cite{delestviennot} in a somewhat different language.

The idea is to read the steps of the upper and lower paths of a parallogram polyomino
$P$ alternatingly and form the Dyck path $D=\ptd(P)$ by using two rules.
We perform a rise of the Dyck path for either a North step of the upper path or an East step of the lower path,
and perform a descent of the Dyck path for either an East step of the upper path or a North step of the lower path.
Using this construction, the polyomino in Figure \ref{fig1} is sent to the Dyck path shown in Figure \ref{fig4}.

\begin{figure}[h]
\includegraphics[scale=0.75]{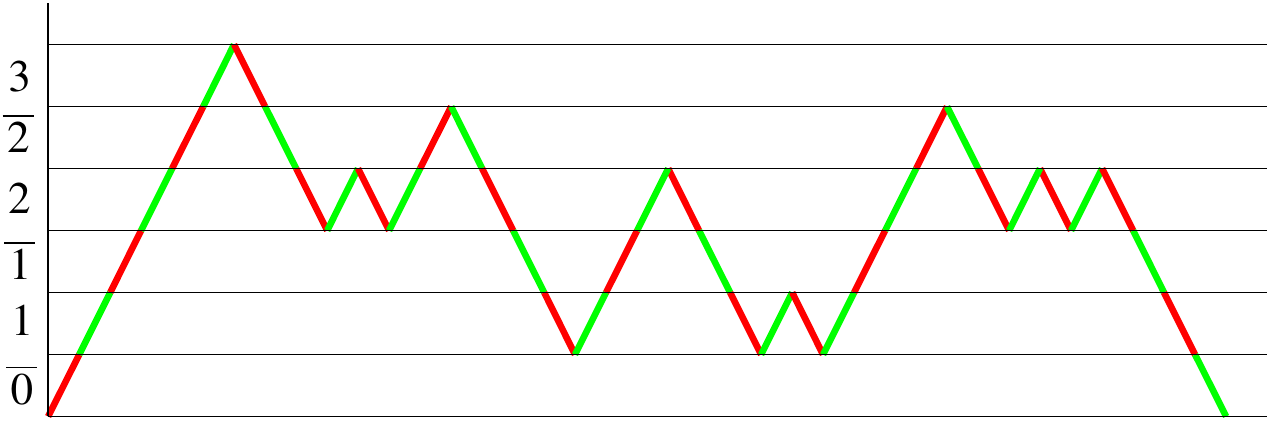}
\caption{\label{fig4}
The Dyck path corresponding to the polyomino of Figure \ref{fig1}.}
\end{figure}

It should be clear that this mapping sends the parallelogram polyomines to the stated subset of Dyck paths.
The fact that the Dyck path does not return to the starting horizontal line before the end corresponds to the
fact that the upper and the lower paths do not intersect each other between the starting and ending points.
The inverse operation is straightforward to describe and verify.

We can easily read the area word of a parallelogram polyomino $P$
from the corresponding Dyck path $\ptd(P)$ as we will now
describe. Consider the Dyck path in Figure \ref{fig4} when reading
the next proposition. It consists of Northeast and Southeast steps
lying between parallel lines which determine certain rows.

\begin{Prop}
Let $P \in \mathrm{Polyo}_{m,n}$ with $D=\ptd(P)$.
If we label the rows of $D$ with
$\overline{0},1,\overline{1},2,\overline{2},3,\overline{3},\cdots$
from bottom to top, then reading the labels of the rows of the
rises from left to right we get the area word of the polyomino $P$.
\end{Prop}

To prove this proposition, we will use induction on the number of
pairs of steps of the upper and lower paths starting form the
Southwest corner. At each step of this induction we will consider
the \textit{partial box} that includes the partial paths, i.e. the
smallest rectangle that includes them (see Figure \ref{fig4_1}).
Then we imagine to complete the paths inside the partial box by
moving along the edges to reach the Northeast corner, and we read
the labels of the resulting polyomino on the partial paths.

\begin{figure}[h]
\includegraphics[scale=0.5]{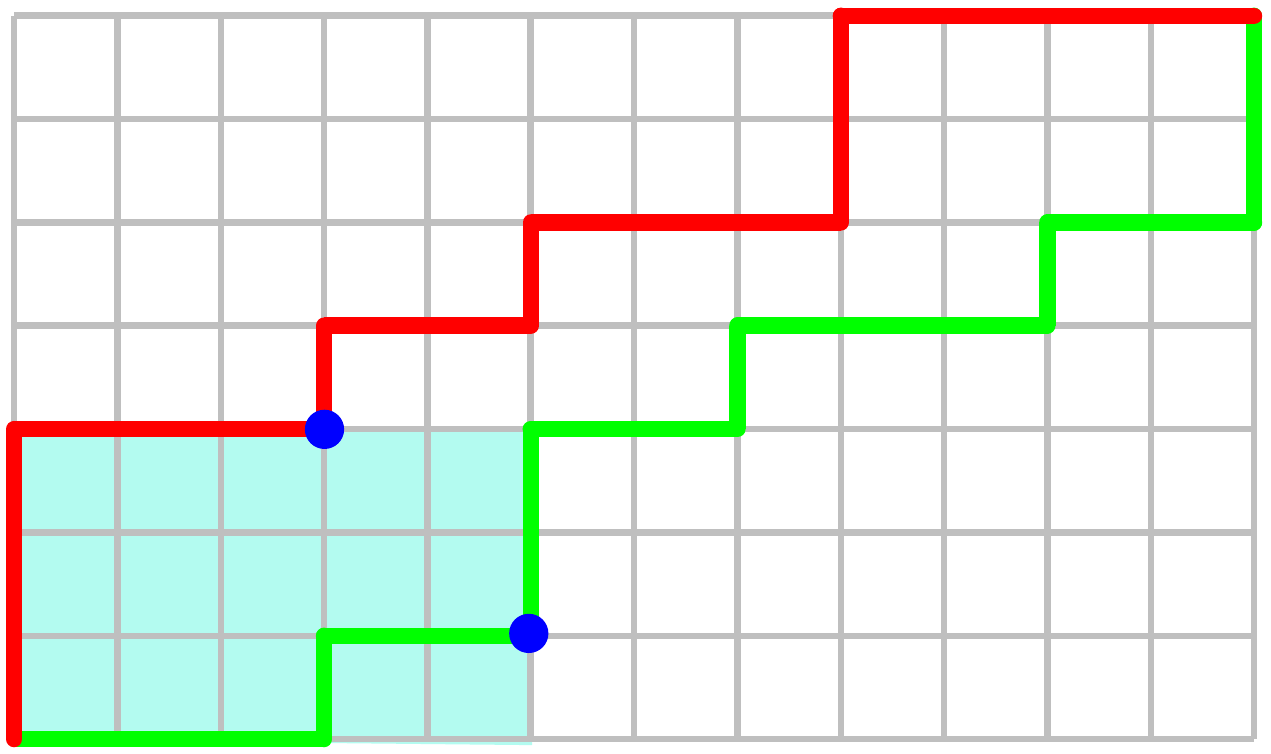}
\caption{\label{fig4_1} The highlighted area is the partial box
after the first $6$ steps of both the green path and the red
path.}
\end{figure}

We claim that this gives exactly the corresponding part of the
area word of the original polyomino.

The key observation is the following claim.
\begin{Claim}
In the last pair of steps, the label of a North step of the upper
path is always the distance from the right edge of the previous
partial box, with a bar on top; while the label of a East step of
the lower path is always the distance from the upper edge of the
partial box.
\end{Claim}
After proving this claim, it remains only to observe that the
distance from the right edge of the previous partial box of the
North steps of the upper path corresponds to the number of odd
rows from the bottom line in the corresponding Dyck path; while
the distance from the upper edge of the partial box of the East
steps of the lower path corresponds to the number of even rows
from the bottom line. This completes the proof of the proposition.
\begin{proof}[Proof of the Claim]
At the beginning the upper path is forced to go North and the
lower path is forced to go East. The partial box at this point
consists of a single square, and we clearly have the partial area
word $\overline{0}1$: this is always the beginning of an area word
for a polyonimo, and it corresponds to the first two rises in the
corresponding Dyck path, as it should be.

Now suppose that everything works up to a certain pair of steps, and let us
make the next pair of steps. We have four cases (see Figure \ref{fig5}):

\begin{figure}[h]
\includegraphics{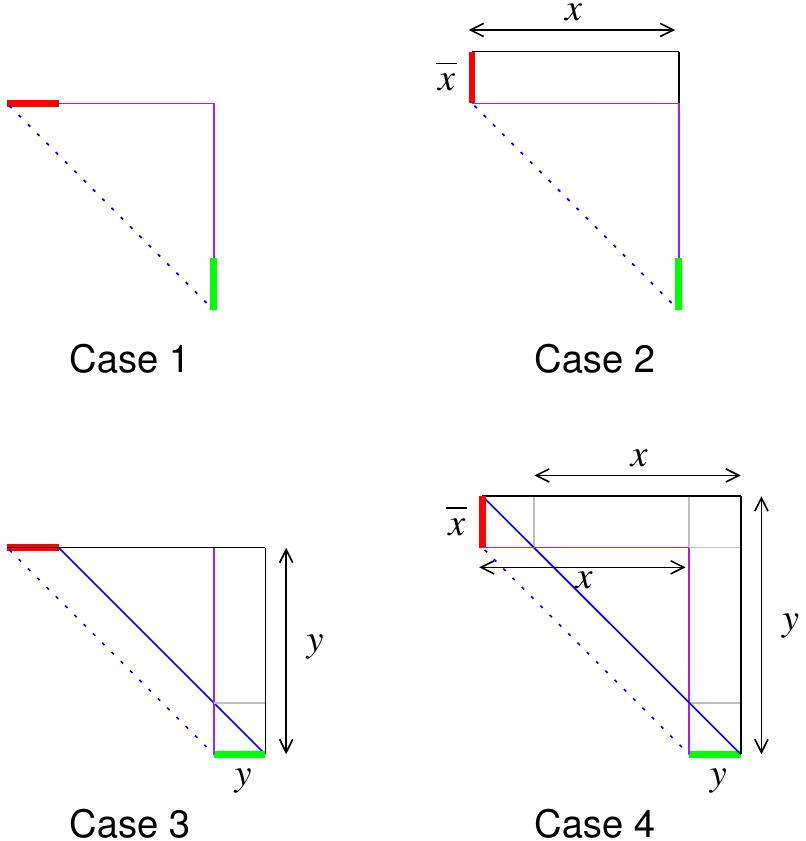}
\caption{\label{fig5}
The four cases. The previous partial box is violet, while
the new one is black.}
\end{figure}

\underline{Case 1}: The upper path moves East, and the lower path
moves North. Then there are no labels to add, and the previous
labels remain unchanged, since the partial box remains unchanged.

\underline{Case 2}: Both the upper and the lower paths move North.
Then the label of the North step of the upper path is clearly the
distance from the right edge of the partial box, which is the same
distance from the one of the previous partial box. The previous
labels clearly remain unchanged.

\underline{Case 3}: Both the upper and the lower paths move East.
Then the label of the East step of the lower path is the
distance to the upper edge of the partial box. The previous
labels of the upper path remain unchanged, since in each row we
added just a box crossed by the diagonal corresponding to the new
East step of the lower path. The previous labels of the lower path
also remain unchanged, since we did not move the upper edge of the
partial box.

\underline{Case 4}: The upper path moves North, and the lower path
moves East. Then the label of the North step of the upper path is
the distance from the right edge of the partial box minus
1, since the first box becomes crossed by the diagonal of the East
step of the lower path. But this is equal to the distance
from the right edge of the previous partial box. The label
of the East step of the lower path is clearly the distance from it to
the upper edge of the partial box. The previous labels of the
upper path remain unchanged, since in each row we added just a box
crossed by the diagonal corresponding to the new East step of the
lower path. The previous labels of the lower path also remain
unchanged, since the diagonals of the previous horizontal steps
all hit the upper path in the same spots as before.
\end{proof}
As an immediate consequence, we get a characterization of the
words in the ordered alphabet
$\overline{0}<1<\overline{1}<2<\overline{2}<3<\overline{3}<\cdots$
which are area words of elements of $\mathrm{Polyo}_{m,n}$.

We state this characterization here as a corollary.
\begin{Cor} \label{cor:charactareawords}
Consider the alphabet
$\overline{0}<1<\overline{1}<2<\overline{2}<3<\overline{3}<\cdots$,
with the letters in the given order. A word $a_1a_2\cdots a_r$ in
this alphabet is the area word of an element of
$\mathrm{Polyo}_{m,n}$ if and only if the following conditions
hold:
\begin{enumerate}
    \item $a_1=\overline{0}$, and this is the only $\overline{0}$ that appears in the word;
    \item there are exactly $m$ of the $a_i$'s which
    are from the set of numbers without a bar $\{1,2,3,\dots\}$,
    and exactly $n$ of the $a_i$'s which are from the set of
    numbers with a bar
    $\{\overline{0},\overline{1},\overline{2},\dots\}$ (in particular $r=m+n$);
    \item for all $i=1,2,\dots,m+n-1$, the letter $a_{i+1}$ is
    less than or equal to the immediate successor of the letter
    $a_i$, in the given order on the alphabet.
\end{enumerate}
\end{Cor}

We mention here that this bijection also gives an easy way to
construct the polyomino from its area word: draw the corresponding
Dyck path (this is immediate), and then look at the odd and even
steps to construct the polyomino.

\section{The bi-statistics $(\mathsf{area},\mathsf{bounce})$ and $(\mathsf{dinv},\mathsf{area})$}
\label{sec:four}

This section is dedicated to proving the following theorem.
\begin{Thm} \label{thm:ADinBA}
For all $m\geq 1$ and $n\geq 1$,
$$\mathsf{Nara}_{m,n}(q,t)=\widetilde{\mathsf{Nara}}_{n,m}(q,t).$$
\end{Thm}

In order to prove this theorem it suffices to give a bijection from
$\mathrm{Polyo}_{m,n}$ to $\mathrm{Polyo}_{n,m}$ which sends the
bi-statistic $(\mathsf{area},\mathsf{bounce})$ to the bi-statistic
$(\mathsf{dinv},\mathsf{area})$.

The bijection that we will now describe is similar in spirit to the one used
in the proof of the analogous \cite[Theorem 3.15]{haglundbook}.

Let $P \in \mathrm{Polyo}_{m,n}$. Starting from $P$, we read the
labels of its bounce path, getting a word consisting of barred and
unbarred numbers. Then, starting from the bottom-left corner, for
each turn of the bounce path, we look at the part of the path
(upper or lower) that includes it. For example in the polyomino of
Figure \ref{fig3}, the first turn of the bounce path is between
$\overline{0}$ and the next $1$ in the labelling of the bounce
path. The containing path consists of the first 4 steps (counted
from the Southwest corner) of the upper path. We label the
vertical steps of the containing path with the labels used for the
vertical steps in that part of the bounce path, and the horizontal
steps of the containing path with the labels used for the
horizontal steps in that part of the bounce path. See Figure
\ref{fig6} for an example.

\begin{figure}[h]
\begin{center}
\includegraphics[scale=0.5]{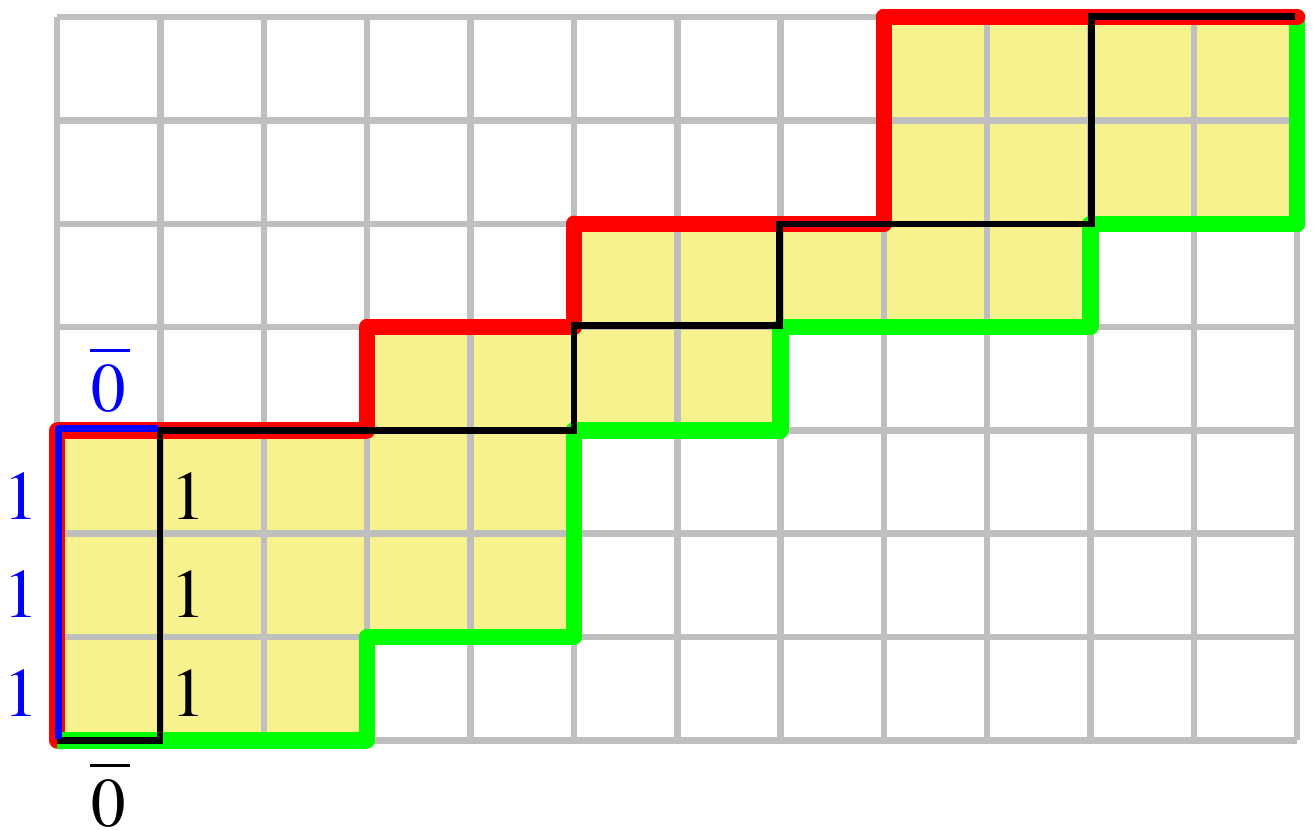}
\caption{\label{fig6}
The containing path and the new labels are blue.}
\end{center}
\end{figure}

We then read the new labels by following the containing path from
Northeast down to Southwest. In the example we read
$\textcolor{blue}{\overline{0}111}$.

During the remainder of the construction we will preserve the relative
positions of these labels.

We then repeat the algorithm with the second turn of the bounce path of $P$.
In the example
this occurs between the last $1$ and the first $\overline{1}$ in
the bounce path. This time the containing path consists of the
steps of the lower path between the second and the eighth. We
repeat the procedure that we used before, and the word that we get
reading the new labels will prescribe the relative positions of
the $1$'s and the $\overline{1}$'s. In the example (see Figure \ref{fig7})
we get the prescriptions
$\textcolor{violet}{11\overline{1}\overline{1}1\overline{1}\overline{1}}$.
This together with the other prescription gives a partial word
$\overline{0}11\overline{1}\overline{1}1\overline{1}\overline{1}$.

In general we will construct this partial word in a way that it
can be the word of a parallelogram polyomino while respecting all the
prescriptions. This will always be possible since the first step
of the containing path that we read will always be labelled by the
smallest of the two types of labels that we are considering: this
is due to the definition of the bounce path.

\begin{figure}[h]
\begin{center}
\includegraphics[scale=0.5]{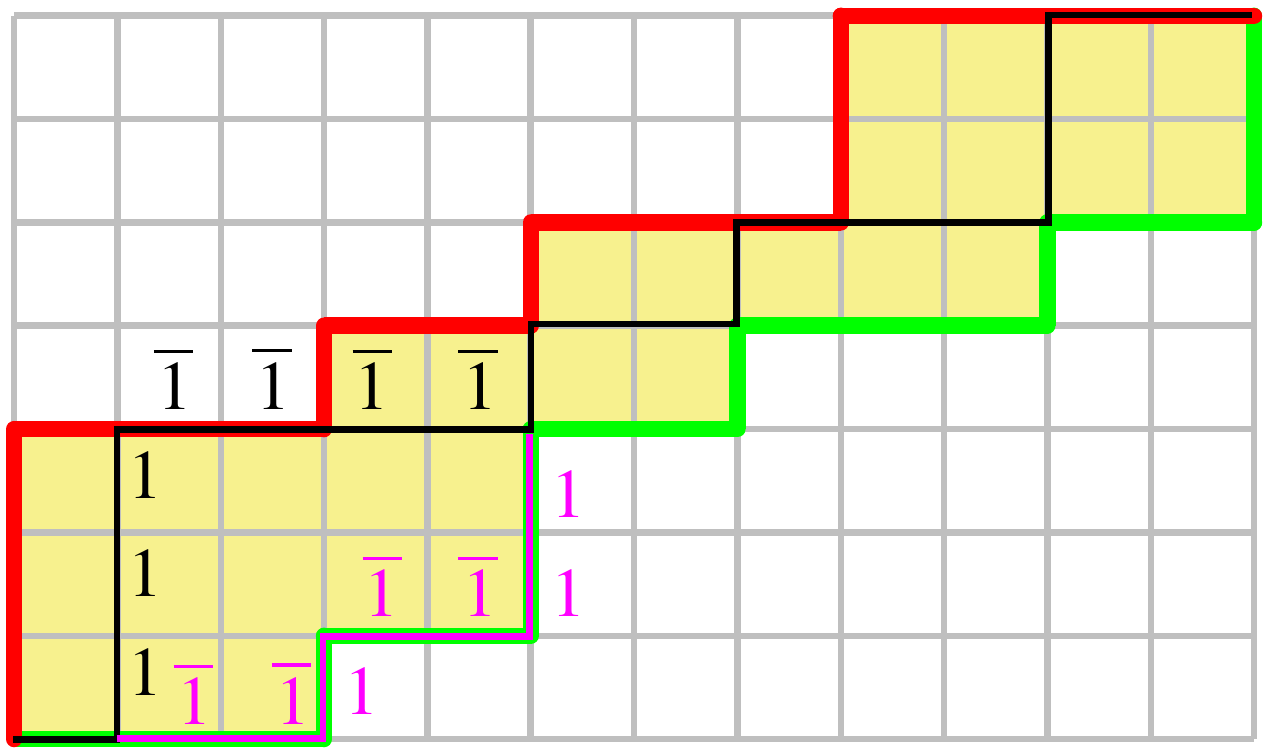}
\caption{\label{fig7}
The containing path and the new labels are violet.}
\end{center}
\end{figure}

We keep doing this until all the labels of the bounce path of $P$ have been included.
At the end we will get a word of another parallelogram polyomino.
In the example, at the next step we get the prescriptions
$\overline{1}\overline{1}2\overline{1}\overline{1}$, which gives
the partial word
$\overline{0}11\overline{1}\overline{1}21\overline{1}\overline{1}$;
then we get the prescriptions $2\overline{2}\overline{2}$,  which
gives the partial word
$\overline{0}11\overline{1}\overline{1}2\overline{2}\overline{2}1\overline{1}\overline{1}$;
then we get the prescriptions $\overline{2}\overline{2}3$,  which
gives the partial word
$\overline{0}11\overline{1}\overline{1}2\overline{2}\overline{2}31\overline{1}\overline{1}$;
then we get the prescriptions
$3\overline{3}\overline{3}\overline{3}$,  which gives the partial
word
$\overline{0}11\overline{1}\overline{1}2\overline{2}\overline{2}3\overline{3}\overline{3}\overline{3}1\overline{1}\overline{1}$;
then we get the prescriptions
$\overline{3}\overline{3}44\overline{3}$,  which gives the partial
word
$\overline{0}11\overline{1}\overline{1}2\overline{2}\overline{2}3\overline{3}\overline{3}44\overline{3}1\overline{1}\overline{1}$;
and finally we get the prescriptions $44\overline{4}\overline{4}$,
which gives the final word
$\overline{0}11\overline{1}\overline{1}2\overline{2}\overline{2}3\overline{3}\overline{3}44\overline{4}\overline{4}\overline{3}1\overline{1}\overline{1}$.

It is clear from the construction and the characterization of
Corollary \ref{cor:charactareawords}, that in this way we get the
area word of a polyomino $\murder(P)$ in $\mathrm{Polyo}_{n,m}$.
Moreover $\murder(P)$ clearly has $\mathsf{area}$ equal to the
$\mathsf{bounce}$ of the original polyomino $P$, again by
construction. Figure \ref{fig8} illustrates $\murder(P)$ for when
$P$ is the polyomino of Figure \ref{fig1}.

\begin{figure}[h]
\begin{center}
\includegraphics[scale=0.5]{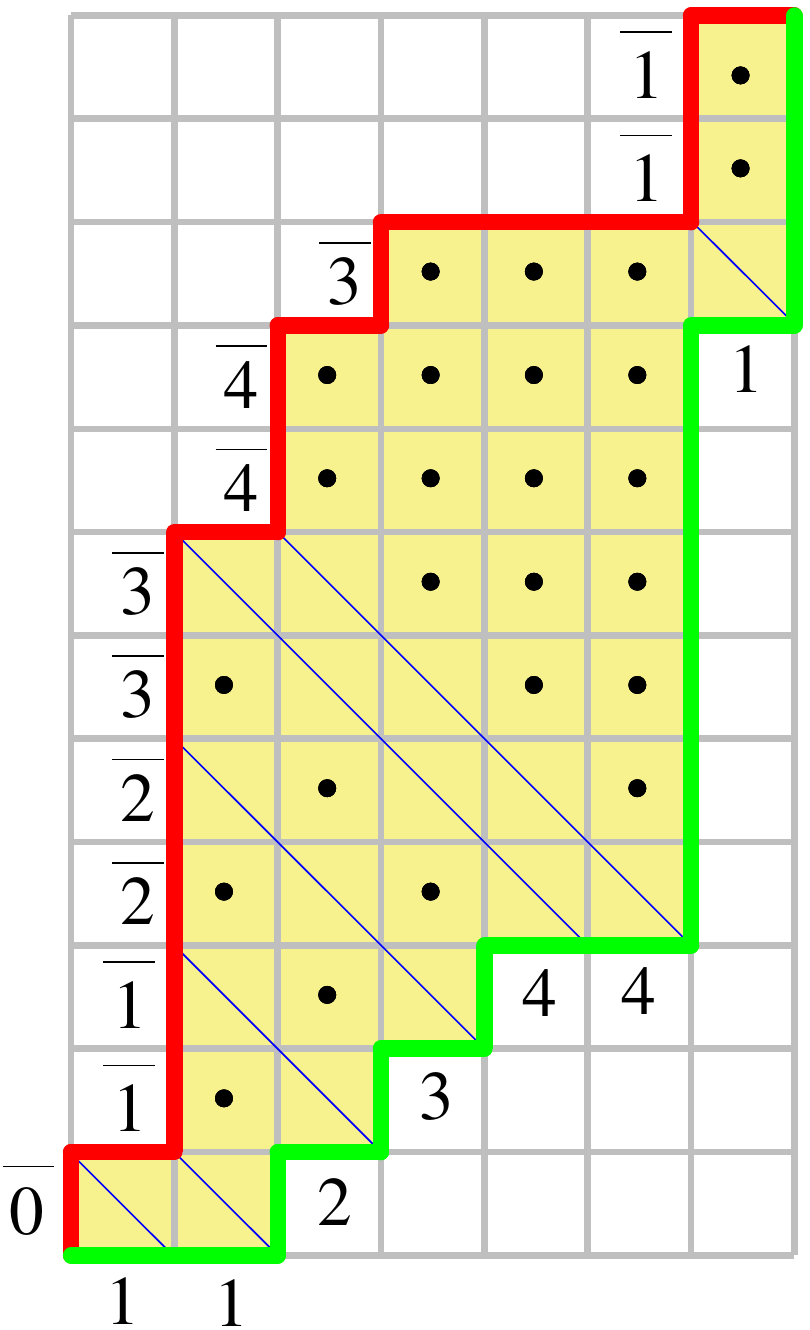}
\caption{\label{fig8}
The outcome of applying $\murder$ to the polyomino of Figure \ref{fig1}.}
\end{center}
\end{figure}

We need to show that the $\mathsf{dinv}$ of $\murder(P)$ is
equal to the $\mathsf{area}$ of $P$.

To see this, recall how we constructed the word of the new
polyomino: for consecutive types of labels, we prescribed the
relative positions by reading the corresponding containing path.
But in the containing path, those pairs of vertical and horizontal
steps which contribute to the $\mathsf{dinv}$ of the polyomino
correspond each to a square in its area.

It remains to show that $\murder$ is a bijection. To see this, we can
consider the inverse function: given a parallelogram polyomino,
write in weakly
increasing order its area word, and draw it as a bounce path with
labels. Then reading the relative positions of consecutive types
of labels you can reconstruct piecewise both the upper and lower
paths. This completes the proof.

Let us observe some remarkable consequences of this result. First
of all, notice that iterating this bijection a second time, we get
a bijection $\murder \circ \murder$ from $\mathrm{Polyo}_{m,n}$ to
itself which sends $\mathsf{bounce}$ to $\mathsf{dinv}$. Moreover,
applying the inverse and composing it with the flip along the
Southwest to Northeast line that pass through the Southwest corner
(which obviously preserves the $\mathsf{area}$) we get a bijection
from $\mathrm{Polyo}_{m,n}$ to itself which sends $\mathsf{dinv}$
to $\mathsf{area}$.

In conclusion, we see that all our three statistics are
equidistributed both inside the same $m$ times $n$ rectangle and
with the polyominoes in the flipped $n$ times $m$ rectangle.

\section{Recursions for $\mathsf{Nara}_{m,n}(q,t)$ and $\widetilde{\mathsf{Nara}}_{n,m}(q,t)$}
\label{sec:five}

In this section we prove that both $\mathsf{Nara}_{m,n}(q,t)$ and
$\widetilde{\mathsf{Nara}}_{n,m}(q,t)$ satisfy a certain
recursion. As an immediate byproduct we get another proof of the
identity
$\mathsf{Nara}_{m,n}(q,t)=\widetilde{\mathsf{Nara}}_{n,m}(q,t)$
stated in Theorem \ref{thm:ADinBA}.

Let $\widetilde{\mathrm{Polyo}}_{m,n}^{(r,s)}$ be the set of
polyominoes in $\mathrm{Polyo}_{m,n}$ whose labelled bounce path
has $r$ many $1$'s and $s$ many $\overline{1}$'s. In other words,
$r$ is the number of steps between the first and the second bounce
of the bounce path, while $s$ is the number of steps between the
second and the third bounce.
Define
$$
\widetilde{\mathsf{Nara}}_{m,n}^{(r,s)}(q,t):=\sum_{P\in
\widetilde{\mathrm{Polyo}}_{m,n}^{(r,s)}}t^{\mathsf{bounce}(P)}q^{\mathsf{area}(P)},
$$
so that $\widetilde{\mathsf{Nara}}_{m,n}(q,t)$ is the sum over all $r$
and $s$ of $\widetilde{\mathsf{Nara}}_{m,n}^{(r,s)}(q,t)$. Also,
we define the $q$-analogue of the non-negative integers by setting
$[0]_q:=1,$
and for all positive integers $n$,
$$
[n]_q:=\frac{1-q^n}{1-q}=1+q+q^2+\cdots+q^{n-1}.
$$
We define the $q$-analogue of the factorial of a non-negative
integer by setting
$ [0]_q!:=1, $
and for all positive integers $n$,
$$ [n]_q!:=\prod_{i=1}^{n}[i]_q.  $$
Finally, for $0 \leq k \leq n$,
$$ {n \brack k}_q :=\frac{[n]_q!}{[n-k]_q! [k]_q!} $$
denotes the $q$-analogue of the binomial
${n \choose k}$.

\begin{Thm}
For all $m,n,r$ and $s$ such that $1\leq r\leq n$ and
$0\leq s\leq m-1$, we have the recursion
$$
\widetilde{\mathsf{Nara}}_{m,n}^{(r,s)}(q,t)=t^{m+n-1}q^{r+s}\sum_{h=1}^{n-r}\sum_{k=0}^{m-s-1}
{s+r-1 \brack s}_q
{s+h-1 \brack h}_q
\widetilde{\mathsf{Nara}}_{m-s,n-r}^{(h,k)}(q,t),
$$
with initial conditions
$$
\widetilde{\mathsf{Nara}}_{m,n}^{(n,s)}(q,t)=\left\{\begin{array}{cc}
  (qt)^{m+n-1}
{m+n-2 \brack m-1}_q
& \text{if $s=m-1$} \\
  0 & \text{if $s<m-1$}, \\
\end{array}\right.
$$
and $\,\,\, \widetilde{\mathsf{Nara}}_{1,n}^{(r,0)}(q,t)=0\,\,\,$
for $\,\,\, r<n$.
\end{Thm}

\begin{proof}
The argument in this proof is best understood by referring to Figure \ref{fig9}.

\begin{figure}[h]
\includegraphics[scale=0.5]{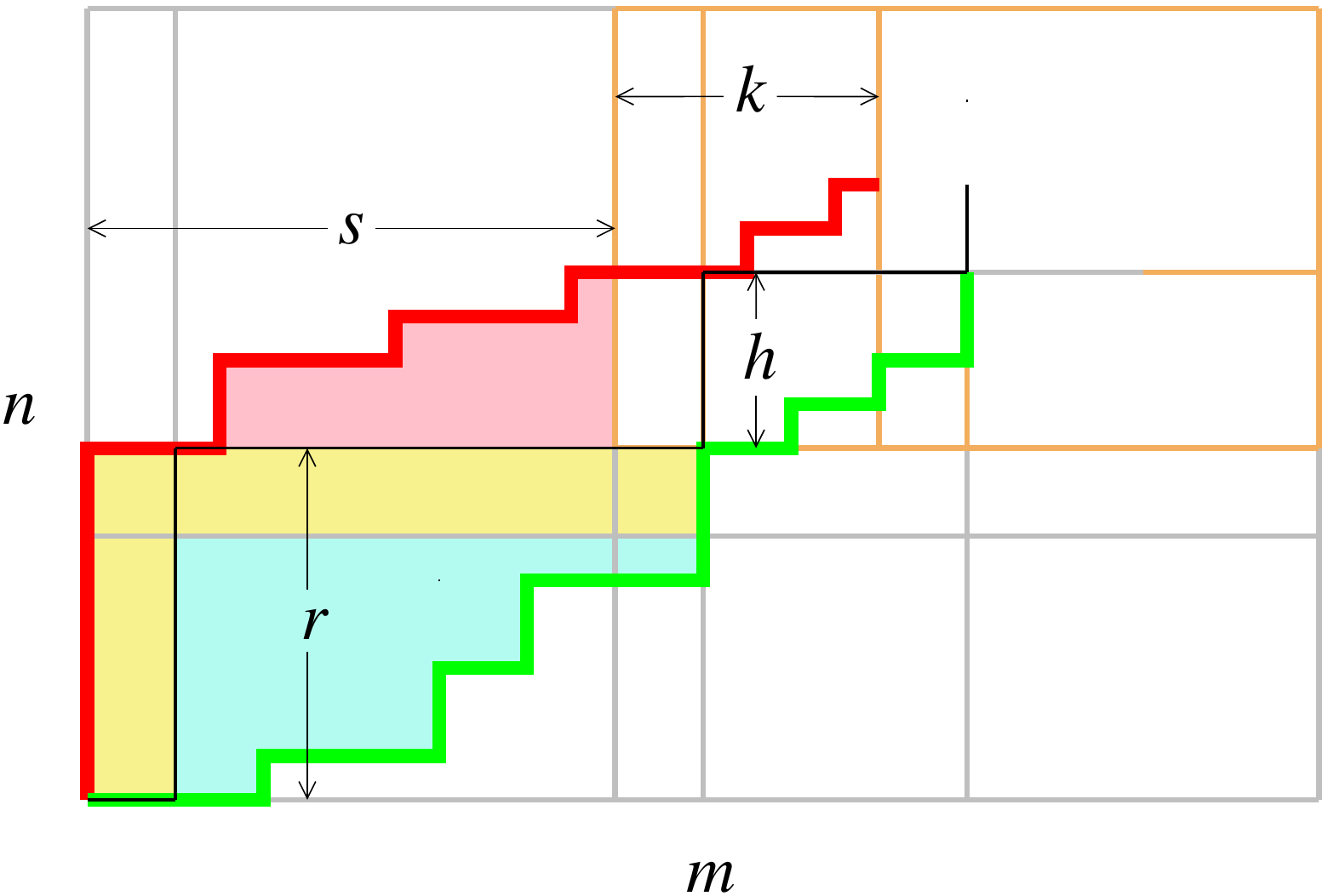}
\caption{\label{fig9}
The black path is the bounce path.}
\end{figure}

The figure shows a typical element of $\widetilde{\mathrm{Polyo}}_{m,n}^{(r,s)}$.
The orange grid cuts
out an element of $\widetilde{\mathrm{Polyo}}_{m-s,n-r}^{(h,k)}$:
its lower-left corner is placed at the beginning of the rightmost
step of the bounce path labelled by $\overline{1}$.

Observe that the labels of the bounce path in the orange grid are
the same as the labels of the corresponding small path all
increased by $1$. Hence, together with the $1$'s and the
$\overline{1}$'s of the bounce path outside of the orange grid, we
see that the bounce of the larger polyomino is $m+n-1$ more than
the bounce of the small polyomino in the orange grid. This shift
is taken care of by the factor $t^{m+n-1}$.

The area of the larger polyomino is equal to
the area of the small polyomino in the orange grid plus the yellow
area, which is taken care of by the factor $q^{r+s}$, the light blue
area, which is counted by the factor
${s+r-1\brack s}_q$, and the pink area, which is counted by the
factor ${s+h-1 \brack h}_q$.
This explains the recursion formula.
\end{proof}

Let us denote by $\mathrm{Polyo}_{n,m}^{(r,s)}$ the set of
parallelogram polyominoes in an $n\times m$ rectangle whose area
word has $r$ many $1$'s and $s$ many $\overline{1}$'s.
Define
$$
\mathsf{Nara}_{n,m}^{(r,s)}(q,t):=\sum_{P\in
{\mathrm{Polyo}}_{n,m}^{(r,s)}}t^{\mathsf{area}(P)}q^{\mathsf{dinv}(P)},
$$
so that ${\mathsf{Nara}}_{n,m}(q,t)$ is the sum over all $r$
and $s$ of ${\mathsf{Nara}}_{n,m}^{(r,s)}(q,t)$.

These polynomials satisfy the same recursion satisfied by the
$\widetilde{\mathsf{Nara}}_{m,n}^{(r,s)}(q,t)$'s:

\begin{Thm} \label{thm:recursion}
For all $m,n,r$ and $s$ with $1\leq r\leq n$ and
$0\leq s\leq m-1$, we have the recursion
$$
\mathsf{Nara}_{n,m}^{(r,s)}(q,t)=t^{m+n-1}q^{r}\sum_{h=1}^{n-r}\sum_{k=0}^{m-s-1}
q^s {s+r-1 \brack s}_q {s+h-1 \brack h}_q
\mathsf{Nara}_{n-r,m-s}^{(h,k)}(q,t),
$$
with initial conditions
$$
\mathsf{Nara}_{n,m}^{(n,s)}(q,t)=\left\{\begin{array}{cc}
  (qt)^{m+n-1}
{m+n-2 \brack m-1}_q & \text{if $s=m-1$} \\
0 & \text{if $s<m-1$}, \\
\end{array}\right.
$$
and $\,\,\, \mathsf{Nara}_{n,1}^{(r,0)}(q,t)=0\,\,\, $ for $\,\,\,
r<n$.
\end{Thm}
\begin{proof}
Given an element of $\mathrm{Polyo}_{n,m}^{(r,s)}$ with $h$ many
$2$'s and $k$ many $\overline{2}$'s, we construct an element of
$\mathrm{Polyo}_{n-r,m-s}^{(h,k)}$ by subtracting $1$ from all the
letters in the area word, then removing all the resulting $0$'s
and $\overline{0}$'s and replacing the only $\overline{-1}$ (which
comes from the only $\overline{0}$) by $\overline{0}$.

For example, if we start with the word
$\textcolor{blue}{\overline{0}}\textcolor{green}{11}\textcolor{red}{\overline{1}}
22\textcolor{red}{\overline{1}}2\overline{2}\overline{2}\textcolor{green}{1}\textcolor{red}{\overline{1}\overline{1}}$,
which is an element of $\mathrm{Polyo}_{6,7}^{(3,4)}$ with $3$
many $2$'s and $2$ many $\overline{2}$'s, then we first get
$\textcolor{blue}{\overline{-1}}\textcolor{green}{00}\textcolor{red}{\overline{0}}11\textcolor{red}{\overline{0}}
1\overline{1}\overline{1}\textcolor{green}{0}\textcolor{red}{\overline{0}\overline{0}}$,
and hence we finally get
$\textcolor{blue}{\overline{0}}111\overline{1}\overline{1}$, which
is an element of
$\mathrm{Polyo}_{6-3,7-4}^{(3,2)}=\mathrm{Polyo}_{3,3}^{(3,2)}$.

Now the area of this new element is clearly $m+n-1$ less than the
area of the original polyomino, since we subtracted $1$ from all
the letters of the area word different from $\overline{0}$. This is
taken care of by the factor $t^{m+n-1}$.

The dinv of the original polyomino is equal to the dinv of this
smaller polyomino, plus the dinv coming from the original
$\overline{0}$ and the $1$'s, which is taken care of by the factor
$q^r$, the dinv coming from the $1$'s and the $\overline{1}$'s,
which is counted by the
factor $q^s {s+r-1 \brack s}_q$
(the $1$'s and the $\overline{1}$'s form a word
which always starts with $1$), and the dinv coming from the
$\overline{1}$'s and the $2$'s, which is counted by the factor
${s+h-1 \brack h}_q$
(as before, the $\overline{1}$'s and the $2$'s
form a word which always starts with $\overline{1}$, but the dinv
coming from this first letter is already counted by the
$\overline{0}$ that we insert in the new area word!).

This explains the recursion.
\end{proof}
As already mentioned, these recursions give immediately
$\mathsf{Nara}_{n,m}^{(r,s)}(q,t)=\widetilde{\mathsf{Nara}}_{m,n}^{(r,s)}(q,t)$,
and hence another proof of the identity
$\mathsf{Nara}_{m,n}(q,t)=\widetilde{\mathsf{Nara}}_{n,m}(q,t)$.

\section{Symmetric functions interpretation}
\label{sec:six}

In this section we will use some tools from the theory of
Macdonald polynomials. For a quick survey of what we need (and
more), we refer to the book \cite{bergeron}, in particular
Chapters 3 and 9.
In what follows we will recall only some basic facts, mostly to fix the
notation.

Let $\Lambda=\bigoplus_{n\geq 0}\Lambda^n$ be the space of
symmetric functions with coefficients in $\mathbb{C}(q,t)$, where
$q$ and $t$ are variables, with its natural decomposition in
components of homogeneous degree.
Recall the fundamental bases of symmetric functions:
\textit{elementary} $\{e_{\mu}\}_{\mu}$, \textit{homogeneous}
$\{h_{\mu}\}_{\mu}$, \textit{power} $\{p_{\mu}\}_{\mu}$,
\textit{monomial} $\{m_{\mu}\}_{\mu}$ and \textit{Schur}
$\{s_{\mu}\}_{\mu}$, where the indices $\mu$ are partitions.

A scalar product is defined on $\Lambda$ by declaring the Schur
basis to be orthonormal:
$$
\langle s_{\lambda},s_{\mu} \rangle=\chi(\lambda=\mu),
$$
where $\chi$ is the indicator function, which is $1$ when its
argument is true, and $0$ otherwise. Another fundamental basis of
$\Lambda$ is $\{\widetilde{H}_{\mu}\}_{\mu}$, the \textit{modified
Macdonald polynomial} basis.

The fundamental ingredient of the theory is the nabla operator
$\nabla$ acting on $\Lambda$. This is an homogeneous invertible
operator introduced by Bergeron and Garsia in the study of the
diagonal harmonics $DH_n$ of $\mathfrak{S}_n$. In fact, it
turns out that $\nabla e_n$ gives precisely the bigraded Frobenius
characteristic of $DH_n$.

The so-called \textit{shuffle conjecture} predicts a combinatorial
interpretation of $\nabla e_n$ in terms of parking functions.
Special cases of this conjecture have been proven by several
authors. In particular, Haglund \cite{haglund} proved the
combinatorial interpretation of $\langle \nabla e_n,h_j h_{n-j}
\rangle$ for $1\leq j\leq n$ predicted by the shuffle conjecture.

Surprisingly, this same polynomial provides the symmetric
functions interpretation of our $q,t$-Narayana numbers.
More precisely, we have the following theorem, which is the main
result of this paper.

\begin{Thm} \label{thm:main}
For $m,n\geq 1$ we have
$$
\mathsf{Nara}_{m,n}(q,t)=(qt)^{m+n-1}\cdot \langle \nabla
e_{m+n-2},h_{m-1} h_{n-1} \rangle.
$$
\end{Thm}
Before proving this theorem, we give here an immediate corollary.
\begin{Thm} \label{thm:symmetries}
For all $m\geq 1$ and $n\geq 1$, we have
$$\mathsf{Nara}_{m,n}(q,t)=\mathsf{Nara}_{m,n}(t,q)$$
and
$$\mathsf{Nara}_{m,n}(q,t)=\mathsf{Nara}_{n,m}(q,t).$$
Moreover, we have
$$\mathsf{Nara}_{m,n}(q,t)=\widetilde{\mathsf{Nara}}_{m,n}(q,t).$$
\end{Thm}

\begin{proof}[Proof of the Theorem \ref{thm:symmetries}]
The symmetry in $q$ and $t$ comes from a general property of the
nabla operator, which is an immediate consequence of the
well-known identity \cite[Equation (9.8)]{bergeron}: nabla applied
to any Schur function is symmetric in $q$ and $t$.

The second equation, symmetry in $m$ and $n$, is obvious from the formula in Theorem \ref{thm:main}.
Finally, the fact that $\mathsf{Nara}_{m,n}(q,t)=
\widetilde{\mathsf{Nara}}_{m,n}(q,t)$ is a direct consequence of
the symmetries and of Theorem \ref{thm:ADinBA}.
\end{proof}

\subsection{Proof of Theorem \ref{thm:main}}
\label{sec:seven}

\newcommand{\onepf}[2]{
\scriptsize{
\setlength{\arraycolsep}{2pt}
\begin{array}{|c|} \hline #1 \\ #2\\ \hline \end{array}
}
}
\newcommand{\pfsize}[1]{\small{#1}}
\newcommand{\fourpf}[8]{
    \pfsize{
    \begin{array}{|c|c|c|c|} \hline #1 & #2 & #3 & #4 \\ #5 & #6 & #7 & #8 \\ \hline \end{array}
    }
}

In order to prove Theorem \ref{thm:main}, we need to make use of Haglund's
combinatorial interpretation of $\langle \nabla e_{m+n-2},h_{m-1} h_{n-1} \rangle$.
To do this we first require some definitions.

For us a \textit{Dyck path of order $k$} will be given by an
\textit{area word} which is a sequence of non-negative integers
$b_1b_2\cdots b_k$ such that $b_1=0$, and $b_{i+1}\leq b_i+1$ for
all $1\leq i <k$.
A {\it{domino}} is a pair of values $(a,b)$ written as the first above the second $\onepf{a}{b}$.

A \textit{parking function} $PF$ of size $k$ is a sequence of $k$
dominoes $\fourpf{a_1}{a_2}{\cdots}{a_k}{b_1}{b_2}{\cdots}{b_k}$
such that $b_1b_2\cdots b_k$ is the area word of a Dyck path of
order $k$, and the $a_i$'s are a permutation of the integers
$\{1,\ldots,k\}$ with the property $a_i<a_{i+1}$ if $b_i<b_{i+1}$
(and hence $b_i=b_{i+1}-1$).

\begin{Ex}
$PF=
\pfsize{\begin{array}{|c|c|c|c|c|c|c|c|c|c|c|} \hline
  5 & 11 & 1 & 9 & 6 & 8 & 3 & 4 & 7 & 10 & 2\\
  0 & 1 & 1 & 2 & 0 & 1 & 0 & 1 & 2 & 3 & 3\\
\hline \end{array}}$ is a parking function of size $11$.
\end{Ex}

\begin{Rem}
Parking functions are often represented by a diagram like the one in Figure \ref{fig10}.
In this diagram the red path represents the underlying Dyck path,
where the number of the squares between the vertical steps of the
Dyck path and the (green) diagonal are given by the lower numbers
in the dominoes. The numbers that label the vertical steps
of the Dyck path in the diagram are simply the upper numbers in
the dominoes.
\begin{figure}[h]
\includegraphics[scale=0.5]{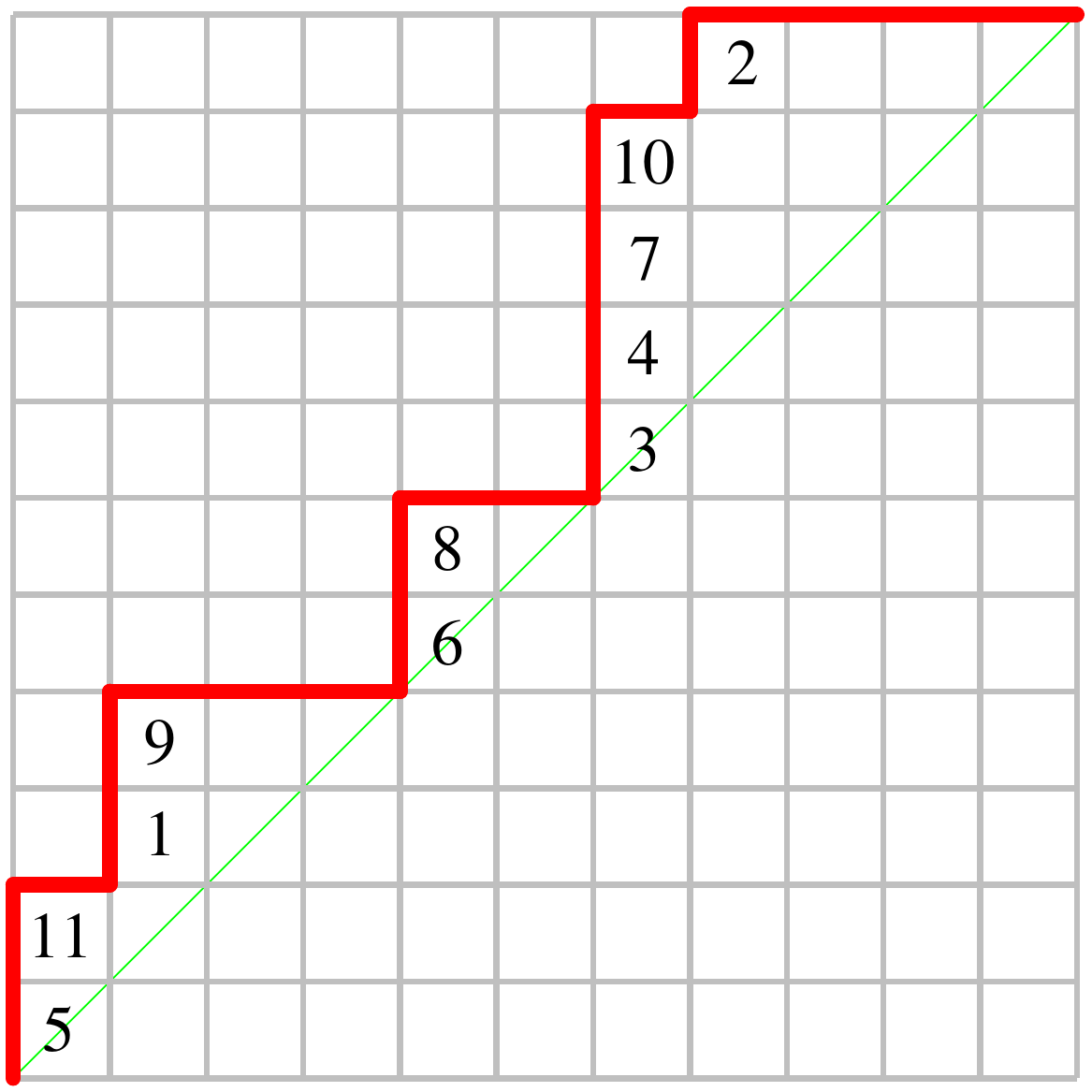}
\caption[ciccia]{\label{fig10}
The parking function
$PF= \pfsize{\begin{array}{|c|c|c|c|c|c|c|c|c|c|c|} \hline
  5 & 11 & 1 & 9 & 6 & 8 & 3 & 4 & 7 & 10 & 2\\
  0 & 1 & 1 & 2 & 0 & 1 & 0 & 1 & 2 & 3 & 3\\
  \hline \end{array}}$.}
\end{figure}
\end{Rem}

Given a parking function, we can reorder its dominoes by comparing
first the bottom numbers, from the biggest to the smallest, and
then, we place the dominoes with the same bottom number in order
as we read them from right to left in the parking function.

The \textit{reading word} $\sigma (PF)$ associated to a parking
function $PF$ is the permutation that we obtain by reading the
upper entries of this reordered sequence of dominoes.

\begin{Ex}\label{hobo}
If
$PF= \pfsize{\begin{array}{|c|c|c|c|c|c|c|c|c|c|c|} \hline
  5 & 11 & 1 & 9 & 6 & 8 & 3 & 4 & 7 & 10 & 2\\
  0 & 1 & 1 & 2 & 0 & 1 & 0 & 1 & 2 & 3 & 3\\
  \hline \end{array}}$ then we reorder the dominoes as
$$\pfsize{\begin{array}{|c|c|c|c|c|c|c|c|c|c|c|} \hline
  2 & 10 & 7 & 9 & 4 & 8 & 1 & 11 & 3 & 6 & 5\\
  3 & 3 & 2 & 2 & 1 & 1 & 1 & 1 & 0 & 0 & 0\\
  \hline \end{array}},$$ and the corresponding reading word is
$\sigma(PF)=2\,\, 10\,\, 7\,\, 9\,\, 4\,\, 8\,\, 1\,\, 11\,\, 3\,\, 6\,\, 5.$
\end{Ex}

Given a parking function $PF=\fourpf{a_1}{a_2}{\cdots}{a_k}{b_1}{b_2}{\cdots}{b_k}$,
we define its \textit{area} to be $\mathsf{area}(PF) =b_1+\ldots+b_k$,
and its \textit{dinv} $\mathsf{dinv}(PF)$ as the number of pairs
$(i,j)$ with $1\leq i<j\leq k $ 
such that either $b_i=b_j$ and $a_i<a_j$, or $b_i=b_j+1$ and $a_i>a_j$.
For example the area of the parking function of Example \ref{hobo}
is $14$, while its dinv is $8$.

Given two disjoint sequences of numbers $A$ and $B$, we denote by
$A \shuffle B$ the set of \textit{shuffles} of $A$ and $B$, i.e.
the sequences consisting of the numbers from $A\cup B$ in which
all the elements of $A$ and $B$ appear in their original
order, so that $|A \shuffle B|= {|A|+|B| \choose |A|}$.

For any $a$ and $b$ in $\mathbb{N}$, we call $\mathrm{Park}_{a,b}$
the set of parking functions $PF$ of size $a+b$ such that
$\sigma(PF)\in (1,2,\dots,a)\shuffle (a+1,a+2,\dots,a+b)$.
Finally, we set
$$
\mathsf{Para}_{a,b}(q,t):=\sum_{PF\in
\mathrm{Park}_{a,b}}t^{\mathsf{area}(PF)}q^{\mathsf{dinv}(PF)}.
$$
We may state now the result of Haglund (see \cite{haglund} for a
proof, and \cite{haglundbook} for the necessary background).

\begin{Thm}[Haglund]\label{Hthm}
For all $m\geq 1$ and $n\geq 1$, we have
$$ \langle \nabla e_{m+n-2},h_{m-1} h_{n-1}
    \rangle=\mathsf{Para}_{n-1,m-1}(q,t).$$
\end{Thm}
This theorem reduces the problem of proving Theorem \ref{thm:main} to proving the following:
\begin{equation}
\label{newequiv}
\mathsf{Nara}_{m,n}(q,t)=(qt)^{m+n-1}
\mathsf{Para}_{n-1,m-1}(q,t).
\end{equation}
In order to show the validity of this equation we do as follows.
For $0\leq r<n$, $0\leq s<m$ with $r+s\geq 1$, let $\mathrm{Park}_{n-1,m-1}^{(r,s)}$ be the set of parking
functions $PF$ of size $m+n-2$ such that
$$
\sigma(PF)\in A\shuffle B, \quad |D_0(PF)\cap A|=r,\quad
\text{and}\,\, |D_0(PF)\cap B|=s,
$$
where $A=(1,2,\dots,n-1)$, $B=(n,n+1,\dots,m+n-2)$, and $D_0(PF)$
is the set of upper numbers of dominoes of $PF$ whose bottom numbers
equal $0$.
Define the polynomial
$$
\mathsf{Para}_{n-1,m-1}^{(r,s)}(q,t):= \sum_{PF\in
\mathrm{Park}_{n-1,m-1}^{(r,s)}}t^{\mathsf{area}(PF)}q^{\mathsf{dinv}(PF)},
$$
and set $\mathsf{Para}_{n-1,m-1}^{(0,0)}(q,t)=0$. Clearly the sum
of all these polynomials as $r$ and $s$ range over their possible
values is equal to $\mathsf{Para}_{n-1,m-1}(q,t)$. With this new
generalization in mind, it is clear that equation \ref{newequiv}
holds true if
$$
(qt)^{m+n-1}\mathsf{Para}_{n-1,m-1}^{(r,s)}(q,t)=\mathsf{Nara}_{m,n}^{(s+1,r)}(q,t).
$$
Our proof of this identity is similar to what Haglund did in
\cite{haglund}.

We will show that
$(qt)^{m+n-1}\mathsf{Para}_{n-1,m-1}^{(r,s)}(q,t)$ also satisfies
the recursion in Theorem \ref{thm:recursion}, with the 4-tuple
$(m,n,r,s)$ replaced with $(n,m,s+1,r)$, i.e.
\begin{align*}
\lefteqn{(qt)^{m+n-1}\mathsf{Para}_{n-1,m-1}^{(r,s)}(q,t)}\\[1em]
&=
t^{n+m-1}q^{s+1}\sum_{k=1}^{m-s-1}\sum_{h=0}^{n-r-1}
q^r
{r+s \brack r}_q {r+k-1 \brack k}_q
(qt)^{m+n-r-s-2}\mathsf{Para}_{n-r-1,m-s-2}^{(h,k-1)}(q,t),
\end{align*}
which simplifies to
\begin{equation} \label{eq:PFrecursion}
\mathsf{Para}_{n-1,m-1}^{(r,s)}(q,t)
=
t^{m+n-r-s-2}\sum_{h=0}^{n-r-1}\sum_{k=1}^{m-s-1}
{r+s \brack r}_q {r+k-1 \brack k}_q
\mathsf{Para}_{n-r-1,m-s-2}^{(h,k-1)}(q,t).
\end{equation}
The initial conditions are
$$
\mathsf{Para}_{n-1,m-1}^{(r,m-1)}(q,t)
=\frac{\mathsf{Nara}_{m,n}^{(m,r)}(q,t)}{(qt)^{m+n-1}}=\left\{\begin{array}{cc}
{m+n-2 \brack n-1}_q & \text{if $r=n-1$,} \\
0 & \text{if $r<n-1$,} \\
\end{array}\right.
$$
and
$$
\mathsf{Para}_{0,m-1}^{(r,s)}(q,t)=\frac{\mathsf{Nara}_{m,1}^{(s+1,0)}(q,t)}{(qt)^m}=0
\text{ for }s<m-1.
$$

In order to see how the recursion \ref{eq:PFrecursion} works for parking functions,
we first make a simplification.
It follows immediately from the
definitions that, since the reading word of the parking functions
we are interested in is a shuffle of the sequences
$A=(1,2,\dots,n-1)$ and $B=(n,n+1,\dots,n+m-2)$, the pairs of
dominoes with both upper numbers in $A$ or both in $B$ do not
contribute to the dinv. The only pairs that contribute are the
ones where one of the upper numbers is in $A$ and the other is in
$B$. Since all the elements of $A$ are smaller than all the
elements of $B$, we can simply consider dominoes in which the
upper number is $1$ (if the corresponding element was in $A$) or
$2$ (if the corresponding element was in $B$), with the dinv
defined in the same way.

For example the parking function $PF\in \mathrm{Park}_{9,9}^{(3,1)}$
$$\pfsize{\begin{array}{|c|c|c|c|c|c|c|c|c|c|c|c|c|c|c|c|c|} \hline
  3 & 13 & 6 & 15 & 8 & 7 & 16 & 12 & 5 & 14 & 9 & 2 & 11 & 1 & 10 & 4 \\
  0 & 1 & 1 & 2 & 2 & 2 & 3 & 1 & 1 & 2 & 0 & 0 & 1 & 0 & 1 & 1 \\
  \hline \end{array}},$$
whose reading word is
$$\sigma(PF)=16\,\, 14\,\, 7\,\, 8\,\, 15\,\, 4\,\, 10\,\, 11\,\,
    5\,\, 12\,\, 6\,\, 13\,\, 1\,\, 2\,\, 9\,\, 3,$$
would correspond to
$$\pfsize{\begin{array}{|c|c|c|c|c|c|c|c|c|c|c|c|c|c|c|c|}\hline
  1 & 2 & 1 & 2 & 1 & 1 & 2 & 2 & 1 & 2 & 2 & 1 & 2 & 1 & 2 & 1\\
  0 & 1 & 1 & 2 & 2 & 2 & 3 & 1 & 1 & 2 & 0 & 0 & 1 & 0 & 1 & 1\\
  \hline \end{array}},$$
whose reading word is
$$\sigma(PF)=2\,\, 2\,\, 1\,\, 1\,\, 2\,\, 1\,\, 2\,\, 2\,\, 1\,\,
    2\,\, 1\,\, 2\,\, 1\,\, 1\,\, 2\,\, 1.  $$
In both cases the dinv is $32$, and the area is $18$.

\textbf{\textsl{Using this identification}}, what we do is the
following: given an element $PF$ of
$\mathrm{Park}_{n-1,m-1}^{(r,s)}$, we remove the dominoes whose
lower number is $0$, and we decrease the lower number of the
remaining dominoes by $1$, keeping them in the given order.

In our example, applying this procedure to $PF$ we get
$$\pfsize{\begin{array}{|c|c|c|c|c|c|c|c|c|c|c|c|c|c|c|c|} \hline
  2 & 1 & 2 & 1 & 1 & 2 & 2 & 1 & 2 & 2 & 2 & 1\\
  0 & 0 & 1 & 1 & 1 & 2 & 0 & 0 & 1 & 0 & 0 & 0\\
  \hline \end{array}}.$$
In doing this, observe that we will always get a parking function which
starts with a domino $\onepf{2}{0}$,
which is always followed by a domino with lower entry $0$, since
$PF$ cannot contain a sequence of three consecutive dominoes
with strictly increasing lower numbers.
This first domino contributes $0$ to both area and dinv and we can therefore remove it.
In doing this we get an element of
$\mathrm{Park}_{n-r-1,m-s-2}^{(h,k-1)}$, where $h$ is the number
of $\onepf{1}{1}$ dominoes in $PF$, and $k$ is the number of
$\onepf{2}{1}$ dominoes in $PF$.

\begin{Rem}

Conversely, given an element of
$\mathrm{Park}_{n-r-1,m-s-2}^{(h,k-1)}$, we can prepend it with a
$\onepf{2}{0}$ domino, increase all the lower numbers by 1, and
then insert $r$ $\onepf{1}{0}$ dominoes and $s$ $\onepf{2}{0}$
dominoes. This gives us an element of
$\mathrm{Park}_{n-1,m-1}^{(r,s)}$. In doing so, we are forced to
put a $\onepf{1}{0}$ in front of the first $\onepf{2}{1}$ which we
just prepended. Other inserted dominoes must satisfy the following
constraints: a $\onepf{1}{0}$ domino is followed by a domino in
$\left\{\onepf{1}{0},\onepf{2}{0},\onepf{2}{1}\right\}$, if any,
and a $\onepf{2}{0}$ domino is followed by a domino in
$\left\{\onepf{1}{0},\onepf{2}{0}\right\}$, if any.
\end{Rem}

Let us now look at how the area and the dinv change with respect to the
operation that we have just described.
The area of the new parking function is equal to the area of $PF$
minus $(n-1-r)+(m-1-s)$, which is taken care of by the factor
$t^{n-r+m-s-2}$ on the right hand side of \eqref{eq:PFrecursion}.

The dinv is going to be the dinv of $PF$ minus the dinv created
by the dominoes that we have removed.
First of all, there are the pairs of dominoes $\onepf{2}{0}$
and $\onepf{1}{0}$ in $PF$, whose relative position creates dinv: this
dinv is taken care of by the factor ${r+s \brack r}_q$
on the right hand side of \eqref{eq:PFrecursion}.
Then there is the dinv created by the dominoes $\onepf{2}{1}$
in $PF$ with the dominoes $\onepf{1}{0}$ in $PF$: this
is taken care by the factor ${r+k-1 \brack k}_q$,
since the first domino $\onepf{2}{1}$ is necessarily preceded by a
$\onepf{1}{0}$.

The initial conditions are obvious.
This completes the proof of Theorem \ref{thm:main}.

\end{document}